\documentclass[amssymb,12pt]{article}

\usepackage{graphicx}
\usepackage{amssymb}
\usepackage{amsfonts,amsthm}
\usepackage{indentfirst}
\usepackage{amsmath}
\usepackage[mathscr]{euscript}
\usepackage{lineno}
\usepackage{tikz}
\makeatletter
\newsavebox\myboxA
\newsavebox\myboxB
\newlength\mylenA
\newcommand*\xoverline[2][0.75]{%
	\sbox{\myboxA}{$\m@th#2$}%
	\setbox\myboxB\null
	\ht\myboxB=\ht\myboxA%
	\dp\myboxB=\dp\myboxA%
	\wd\myboxB=#1\wd\myboxA
	\sbox\myboxB{$\m@th\overline{\copy\myboxB}$}
	\setlength\mylenA{\the\wd\myboxA}
	\addtolength\mylenA{-\the\wd\myboxB}%
	\ifdim\wd\myboxB<\wd\myboxA%
	\rlap{\hskip 0.5\mylenA\usebox\myboxB}{\usebox\myboxA}%
	\else
	\hskip -0.5\mylenA\rlap{\usebox\myboxA}{\hskip 0.5\mylenA\usebox\myboxB}%
	\fi}
\makeatother

\newcommand{\keywords}[1]{\par\noindent\textbf{Keywords:} #1}

\usepackage{authblk}

\newtheorem{thm}{Theorem}[section]
\newtheorem{lem}[thm]{Lemma}
\newtheorem{pro}[thm]{Proposition}

\newtheorem{dfn}[thm]{Definition}

\newtheorem{rk}[thm]{Remark}

\newtheorem{claim}[thm]{Claim}

\newcommand{\norm}[1]{\left\Vert#1\right\Vert}

\def\cal#1{\fam2#1}

\def\bb{\begin}
\def\be{\begin{equation}}
	\def\ee{\end{equation}}
\def\bea{\begin{eqnarray}}
	\def\eea{\end{eqnarray}}
\def\beaa{\begin{eqnarray*}}
	\def\eeaa{\end{eqnarray*}}

\def\ifl{\iffalse}

\def\bb{\begin}
           \def\ea{\end{array}}
          \def\ec{\end{center}}
     \def\ed{\end{description}}
\def\be{\bb{equation}}        \def\ee{\end{equation}}
\def\bea{\bb{eqnarray}}       \def\eea{\end{eqnarray}}
\def\beaa{\bb{eqnarray*}}     \def\eeaa{\end{eqnarray*}}
 \def\et{\end{thebibliography}}

\def\bar2{\doublebar}

\begin{document}

\title{On the phenomenon of topological  chaos and statistical triviality}

\author[a]{Chao Liang} 
\author[b]{Xiankun Ren \thanks{ Corresponding author. \\Email address: chaol@cufe.edu.cn (C. Liang), xkren@cqu.edu.cn (X. Ren), sunwx@math.pku.edu.cn (W. Sun), 
		vargas@ime.usp.br (E. Vargas)}}
		\author[c]{Wenxiang Sun}
		\author[d]{Edson Vargas }
\affil[a]{School of Statistics and Mathematics, Central University of Finance and Economics, Beijing 100081, People's Republic of China}
\affil[b]{College of Mathematics and Statistics, Chongqing University, Chongqing 401331, People's Republic of China}
\affil[c]{School of Mathematical Sciences, Peking University, Beijing 100871, People's Republic of China}
\affil[d]{IME, University of Sao Paulo, Sao Paulo 05508-090, Brazil}
\date{}

\maketitle

\begin{abstract}
There exists a compact manifold so that the set of topologically chaotic but statistically trivial $C^{r} (1\leq r \leq \infty)$  vector fields  on this manifold 
displays considerable scale in the view of dimension. More specifically, it contains an infinitely dimensional connected subset.
\end{abstract}
\section{Introduction}

In the field of dynamical systems, both topological complexity and statistical complexity play important roles in understanding the behavior and properties of dynamical systems, but they address different aspects of complexity.

{\bf Topological complexity}  often refers to the intricacy of the phase space structure, including the presence of attractors, their basins of attraction, and the topology of trajectories. It is often characterized by utilizing topological entropy and certain trajectory properties such as minimality, topological mixing, topological transitivity and others.

{\bf Statistical complexity}  typically involves quantifying the degree of unpredictability or disorder within the system. It is often characterized by utilizing metric entropy of the partitioned representation, Lyapunov exponents and some properties about the measures supported on the system.

While topological complexity focuses on the geometric or qualitative structure of phase space dynamics, statistical complexity captures the statistical properties and information content of dynamical trajectories. These measures are often used {\bf in tandem} to gain a more comprehensive understanding of the dynamics of complex systems and in many cases, the presence of one type of complexity can lead to or coincide with the presence of the other due to the underlying dynamics of the system, leading to a correlation or interplay between the two.

In \cite{Herman}, Herman  constructs a real analytic diffeomorphism on a compact connected 4-dimensional manifold, such that it  is  minimal and its
metric entropy is positive (hence its topological entropy is  positive ). The Lorenz attractor and the H$\acute{e}$non map exhibit a similar correlation: it is both topologically complex and statistically complex.

However, in \cite{S-Y-Z} Sun-Young-Zhou constructed two equivalent ~ $C^r\,\,(r\geq 1)$~ smooth flows with a singularity, one of which has positive topological entropy while the other has zero topological entropy.
The flow with zero entropy is topologically transitive (topologically complex) and, at the same time, its unique   ergodic measure is  supported on a singularity (statistically trivial), which contradicts the correlation mentioned above. In \cite {H-Z} He-Zhou constructed a topological mixing (topologically complex)  homomorphism for which  the unique   ergodic measure is  supported on a fixed point (statistically trivial).
We will deal with the interesting  phenomenon  of  topological complexity and statistical triviality  
in the setting of  $C^r$  vector fields in this article.

Let $M$ be a compact Riemannian manifold and let  $\phi: M \times \mathbb{R} \rightarrow M $ be a continuous flow,  that is a continuous map satisfying ~$\,\, \phi(x, 0)=x$~ and ~$\,\,\phi(\phi(x, t_1), t_2)=\phi(x, t_1+t_2),$ for all  ~$ x\in M,\,\,t_1, t_2\in \mathbb{R}. $  
For~ $t\in \mathbb{R}$~ denote by $\phi_{t}: M \rightarrow M$ the  homeomorphism given by $\phi_{t}(x) = \phi(x, t),  x\in M$.
We call $\phi: M \times \mathbb{R} \rightarrow M$ a ~$C^r$~flow provided the map ~$\phi$~ is ~$C^r,$~ $r = 1, 2, \cdots,+\infty$.
A Borel probability measure $\mu$ is called $\phi_{t}$-{\itshape invariant} if for any Borel set $B$ one has $\mu(B) = \mu(\phi_t(B))$.   $\mu$ 
is called  $\phi$-{\itshape invariant}   if it is $\phi_t$-invariant for all $t\in\mathbb{R}$. A $\phi$-invariant measure $\mu$ is called {\it ergodic } with respect to $\phi$ if for any $\phi$-invariant Borel set $B$ it holds that $\mu(B)=0$ or $\mu(B)=1.$ Denote   $\mathcal{M}_{erg, \phi}$ the set of $\phi$-invariant  and $\phi$-ergodic Borel probabilty measures. 
We use \{{\itshape  fixed points}\}
to denote both the set of  fixed points of $\phi$  and the set of the atomic probability measures supported on fixed points. 
And we use the terminology \{{\itshape  singularities}\} similarly for a $C^r$ differential flow.

\begin{dfn}\label{def.top chaos and stat trivial}
	Let $\phi :\, M \times\mathbb{R} \to M$ be a continuous flow on a compact Riemnannian manifold. We say that $\phi$ is {\bf topologically chaotic but statistically trivial}, if  it is topologically transitive meaning the existence of  a dense orbit and it
	has a single fixed point ~$p$~ such that  $\mathcal{M}_{\phi, erg}= $\{{\itshape  fixed point  p}\}.~ 
\end{dfn}

There could be  various ways to define  topological chaos  but statistical triviality, among them is to define it  by using topological mixing together with zero entropy. What situation  we
deal with in the article is specifically
Definition \ref{def.top chaos and stat trivial}. A natural problem is the following.

{\it Whether the system  with topological chaos and statistical triviality is an exceptional example or such ones   display considerable scale in the view of dimension ? }

\bigskip

We denote by ${\cal{X}}^r(M)$ the set of all $C^r$ vector fields on $M$, and denote by ${\cal{T}}^r(M)$ the subset of ${\cal{X}}^r(M)$ consisting of all $C^r$ vector fields whose induced flows are topologically chaotic but statistically trivial, where $r = 1, 2, \cdots, +\infty$. One conjecuture
of Palis \cite{P} says the following.

\begin{quote}
	{\it 	There exists a $C^1$ open and dense subset $\mathbb{O}\subset {\cal{X}}^1(M)$
		such that any vector field in $\mathbb{O}$ either is Morse-Smale or has a transverse homoclinic intersection.}
\end{quote}

The $C^1$ diffeomorphism version of this conjecture has been proved by Pujals-Sambarino \cite{P-S} in dimension 2, by Bonatti-Gan-Wen \cite{B-G-W} in dimension 3 and by Crovisier \cite{C} in higher dimension. Note that the presence of a transverse homoclinic intersection implies the existence of the horseshoe, which in turn indicates that the system possesses positive topological entropy. And Morse-Smale systems are either Anosov, thus possessing positive entropy, or non-Anosov, thereby lacking dense orbits. Thus the systems in ${\cal{T}}^r(M)$  are openly-and-densely excluded from this conjecture when $r=1$.

However, we will show in this paper that there exists a manifold $M$ so that  ${\cal{T}}^r(M)$ displays considerable scale in the view of dimension, where ~$r\geq 1$~. 

\begin{thm}[Main Theorem]\label{Main thm}
There exists a compact smooth manifold $M$ with the following property.  
For $ r=1,2,\cdots,+\infty$  there is ${\cal{A}}\subseteq  {\cal{T}}^{r}(M)$ such that
\begin{itemize}
\item[(1)] Each $Y\in\cal{A}$ is topologically chaotic but statistically trivial;
\item[(2)] $\cal{A}$ is connected and  the dimension of $\cal{A}$ is infinite. 
\end{itemize}
\end{thm}


\section{Preliminaries}

In this section, we prepare some notions and lemmas.

\subsection{Attracting center}

\begin{dfn}
	A subset $C \subseteq M$ is called an {\bf  attracting center} for a $C^r$ flow $\phi:\, M \times\mathbb{R}\to M$, if
	$$\lim_{\delta\to0}\liminf_{t\to+\infty}\,\frac{1}{t}\int_{0}^{t}\chi_{B(C,\delta)}(\phi_sx) ds = 1,\,\forall x\in M,$$
	where $B(C, \delta) = \{y \in M \,\, |\,\,  \exists \ x\in C \ s.t., d(x, y) < \delta\}$ and $\chi$ is the characteristic function of a set. 
\end{dfn}

\begin{dfn}
	An attracting center $C$ is {\bf minimal} if
	
	\begin{enumerate}

		\item[(1)] it is compact, and
		\item[(2)] $\phi_t (C)=C,\,\forall t\in\mathbb{R}$ and
		\item[(3)]  there is no proper subset of $C$ satisfying both $(1)$ and $(2)$.
	\end{enumerate}
	
\end{dfn}

The concept of minimal attracting center was introduced by Beboutov and
Stepanov in \cite{B-S}. The following lemma is  one characterization of minimal attracting center which was given by Zhou \cite{zhou}.  

\begin{lem}(\cite{zhou})\label{lem.zhou}
	For a $C^r$ flow $\phi:\, M\times\mathbb{R}\to M$ on a compact manifold $M$, the minimal attracting center  $C$ coincides with the closure of the union of supports of $\phi-$invariant ergodic measures, i.e.
	$$C = closure\left\{\bigcup_{\mu\in{\cal{M}_{{\text{erg}},\phi}}} Supp(\mu)\right\}.$$
\end{lem}

\subsection{Time reparameterization}

Two continuous flows  $\phi,\,\,\psi: M\times \mathbb{R}\to M$
are equivalent if there exists a homeomorphism ~  $\pi: M\to M$~  preserving  time orientation  such that

$$\pi (Orb(x, \phi))=Orb(\pi(x), \psi),\quad \forall x\in M.$$                 Two equivalent continuous flows
 can be expressed in term of each other through time reparameterization.

\begin{lem}\label{lem:1}
Let  two continuous  flows $\phi,\,\,\psi: M\times \mathbb{R}\to M$ be equivalent. 
	Set $M_0=\{fixed \,\, points \,\,of\,\, \phi \}.$ Then there exists a
	continuous function $\theta (x,t),\,\,x\in M\setminus M_0,\,\, t\in
	\mathbb{R}$ such that
	\begin{enumerate}
		\item $\theta(x,0)=0$ and $\theta(x,\cdot): \mathbb{R}\to \mathbb{R} $ is
		strictly increasing for any $x\in M\setminus M_0$;
		\item  $\theta(x, s+t)= \theta(x, s)+\theta(\phi_{s}(x), t);\ \forall s,t \in \mathbb{R},\ \forall x\in M\setminus M_0$; 
		\item $\pi \circ \phi_t(x)=\psi_{\theta(x, t)}\circ \pi(x), \ \forall t \in \mathbb{R},\ \forall x\in M\setminus M_0$.
	\end{enumerate}

\end{lem}
\begin{proof}
	We refer to  \cite{ Rohlin} for a proof.
\end{proof}

We call $\theta(x,t): M\setminus M_0 \times \mathbb{R} \rightarrow \mathbb{R}$ a {\it reparameterized time} while transferring $\phi$ to $\psi$ and intuitively
call $\pi$ the {\it transfer map}.

On the other side,  given a flow defined on $M$ and a reparameterized time function $\theta(x,t)$, we can get a related flow.
Let  $\phi: M\times \mathbb{R}\to M$ be a continuous flow 
and let $M_0$ denote the set of
fixed points of $\phi$. Let $\theta (x,t),\,\,x\in M\setminus M_0,\,\,
t\in \mathbb{R}$ be a continuous function satisfying the following:

\begin{enumerate}
	\item $\theta(x,0)=0$ and $\theta_x=\theta(x, .): \mathbb{R}\to
	\mathbb{R} $  is strictly increasing $\forall x\in M\setminus M_0$;   
	
	\item $\theta(x, s+t)= \theta(x, s)+\theta(\phi_{s}(x), t),$ $\forall
	s,\,\, t\in \mathbb{R},\,\,\forall x\in M\setminus M_0.$
\end{enumerate}
Define a flow  $\psi : M\times \mathbb{R}\to M,$
\begin{gather}
		  \psi_{\theta(x,t)}(x) = \phi_{t}(x) ,\quad \forall x\in M\setminus M_{0},\quad t\in \mathbb{R}, \nonumber\\
	\psi_t(x)=x,\,\,\forall x\in M_0,\,\,\forall t \in \mathbb{R}. \nonumber
\end{gather}
We call
$\psi$ to be a {\itshape time-changed flow} from $\phi$ by $\theta(x,t)$. The flows  $(M,
\phi)$ and $(M, \psi)$ are clearly equivalent with transfer
map $\pi:=id: M\to M$. If further $\phi   $ is  $C^r$  $(r\geq 1)$  flows induced by vector   fields $X,$  the function  $\theta(x,
t),\,\,x\in M\setminus M_0,\,\,t\in \mathbb{R}$ is $C^r$
differentiable in both variables, then we call $\psi $   a {\itshape natural
time-changed flow} from $\phi$ by $\theta(x,t). $ In this case there
is a $C^r$ vector field $Y$ on $M$ such that $\psi$ is induced by
$Y.$
The following proposition provides an integral-type description for the reparameterized time function $\theta(x,t)$.
\begin{pro}(\cite[Proposition 4.1]{S-S-V})\label{pro:3}  For a natural time-changed flow from $\phi$ to $\psi$ by
	$\theta(x,t)$, it holds
	$$\theta(x, t)=\int_0^t \frac{\| X(\phi_sx)\| }{\|Y(\phi_s x)\|}\, ds, \,\,\,\, \forall t\in \mathbb{R},\,\,\,\,x\in M\setminus M_0,$$
where $  M_0=\{  \text {Singularities of }  \phi     \}.$

\end{pro}



\section{Proof of Theorem\ref{Main thm}}

We divide the proof into two steps.   \\

{\bf Step 1} Costruction of a manifold and a $C^{r}$ flow on it which is topologically chotic but statistically trival.

Let  $M_1 = SL(2,\mathbb{R})/\Gamma$ with normalized Lebesgue measure $\sigma$, where 
$\Gamma $ is a compact discrete subgroup of $SL(2,\mathbb{R})$. Let ~$\mathbb{T}  $~denote the unit circle  and  consider an analytic    diffeomorphism
\begin{align*}
	F_{\chi}: \mathbb{T} \times M_1 &\rightarrow   \mathbb{T} \times M_1 \\
	(\theta, \beta) &\mapsto (\theta+\chi, \,  A_{\theta}\cdot \beta),
\end{align*}
where $A_{\theta} = \begin{bmatrix} \cos 2\pi \theta & \sin 2\pi \theta \\ \sin 2\pi \theta & \cos 2\pi \theta\end{bmatrix} \cdot \begin{bmatrix}
	\lambda & 0 \\ 0 & 1/\lambda\end{bmatrix}$ and  $\lambda>1$  is  fixed. 
Herman showed  in  \cite{Herman} that there is $\chi_0\in \mathbb{T}$ such that $F_{\chi_{0}}$ is minimal and furthermore the metric entropy of $F_{\chi_{0}}$ with respect to $\sigma$ is strictly positive. By the variational principle, the topological entropy is also strictly positive. 
Denote $M= \mathbb{T} \times M_1  $ and  $ f = F_{\chi_{0} }: M  \rightarrow M.$
Define $$\Omega= M   \times [0,1]/ \sim,$$ where $\sim$ is the equivalence  relation on $M\times  [0,1] $  identificating  $(y,1)$ with $(f(y),0).$  
We mention that $dim( M_1) = 4$ and $dim(\Omega)=5$.

The {\itshape standard suspension}  of $f$  is the flow $\psi$  on $ \Omega$ defined by 
$\psi_{t}(y,s) = (y,t+s)$, provided $0\leq t+s < 1.$ 
Then  $\Omega$ is a  compact Riemannian manifold and $\psi$ is an analytic flow.  Since $f$ is minimal, the flow $\psi$ is also a minimal flow.  Denote by $X$ the vector field on $\Omega$  which induces $\psi$. For any function $\alpha\in C^{\infty}(\Omega,[0,1])$, note that $\alpha X$ is a $C^{\infty}$ vector field on $\Omega$ and it thus induces a differentiable flow on $\Omega$.  Throughout the rest of the proof, we fix one point $p = (x_0,0) \in \Omega$ and consider a function $\alpha$ satisfying the following criteria.\\

{\bfseries (H).} Assume that $\alpha$ satisfies that 
\begin{itemize}
	\item[1)] $\alpha(p) = 0$ and $\alpha(q) >0$ for $q\in \Omega\setminus\{p\}$;
	\item[2)]  There is a small neighborhood $V\subset \Omega$ of $p$ such that $\alpha(q)  \equiv 1$ for any $q \in \Omega\setminus V$.
\end{itemize}

For $\epsilon>0$, we will denote balls (centered at $x$) in $M$ and in $\Omega$  by $$B_{M}(x,\epsilon) = \{y\in M\,\,|\,\, d(x,y) < \epsilon\}$$ and  $$B_{\Omega}(x,\epsilon) = \{y\in \Omega \,\,|\,\, d(x,y)< \epsilon\},$$ respectively. And denote the ball in $\mathbb{R}^{n}$ at the origin with radius $r>0$ by 
$$ B^{n}(r) = \{x\in \mathbb{R}^n\,\, |\,\, \|x\| <r\}.$$

We will construct a specific  $\alpha \in C^{\infty}(\Omega,[0,1])$ such that the flow induced  by $\alpha X$ will be transitive and will have only  one ergodic measure,  which is  atomic.

The following lemma is \cite[Lemma 2.4]{S-Y-Z} which determines  lower bounds on the measures of sets with respect to an ergodic measure in a minimal system.
\begin{lem}\label{lenth}
	Let $M$ denote a smooth compact Riemannian manifold and $f : M \rightarrow M$ be a  minimal homeomorphism.
Suppose $\mu$ is an ergodic measure of $f$. Then for any $\epsilon>0$, there
	is $L(\epsilon) \in \mathbb{N}$ such that 
	\[
	\mu(B_{M}(x,\epsilon)) \geq \frac{1}{L(\epsilon)}, \ \forall x\in M, \,\, \epsilon>0.
	\] 
\end{lem}



Let ~$$\pi:  M\times \mathbb{R}\to \Omega= M   \times [0,1]/ \sim  $$  denote the quotient map.  
The lemma below plays an important role in the construction of $\alpha$.
\begin{lem}[Lemma 3.1 of \cite{S-Y-Z}] \label{exifunofpr}
	For a given sequence of positive numbers 
	$$1 = \beta_{-1} > \beta_0> \beta_1 >  \beta_2 >  \cdots,$$ there is a $C^{\infty}$ function $\omega: B^{n}(2) \rightarrow [0,1]$ such that 
	\begin{itemize}
		\item[1)]  $\omega(x) = 0$ if and only if $x=0$;
		\item [2)] $|\omega|_{B^{n}(\frac{1}{i+1})}|_{C^{0}} \leq \beta_{i-1},\ i=0,1,2.\dots;$
		\item [3)] $\omega|_{B^{n}(2)\setminus B^{n}(1)} =1$.
	\end{itemize} 
\end{lem} 

Recall that the point $p =(x_0,0)$ is a point fixed in condition ({\bfseries H}). Without loss of generality, we can assume the existence of a coordinate chart $(\tilde{V}, \xi)$
of $\Omega$ satisfying the following:
\begin{enumerate}
	\item There is an open set $V \subset \Omega$ such that $p\in V$ and $\overline{V} \subset \tilde{V}$.
	\item  $\xi(p) = 0,\ \xi(V) = B^{5}(1), \ \xi(\tilde{V}) = B^{5}(2)$.
	\item There is $i_1 \in \mathbb{N}$ such that $$\overline{\pi(B_{M}(x_0,1/i_1)\times \{0\})} \subset V$$
	and $$\xi(\pi(B_{M}(x_0, 1/i_1)\times \{0\})) \subset \mathcal{R} = \{x = (x_1,x_2,x_3,x_4,x_5) : x_5 =0\}.$$
	\item There is $i_2\in \mathbb{N}$ such that $B_{\Omega}(p,1/i_2) \subset V$
	and $\xi(B_{\Omega}(p,1/i) ) = B^{5}(1/i)$ for any integer $i> i_2.$ 
\end{enumerate}

Under these assumptions, there is an integer $i_3 >i_2$ with the property that for any $i \geq 0$ there is $0 < l_{i_3 + i} \ll 1$ such that
\[
\overline{\pi\left(B_{M}(x_0,\frac{1}{i_3 +i})\times [-l_{i_3+i},0])\right)} \subset B_{\Omega}(p, \frac{1}{i_2 +i}).
\]

Set $i_0 = \max\{i_1, i_2, i_3\}$. For any $i > i_0$, by Lemma \ref{lenth}, there exists $L(\frac{1}{i})$ such that for any ergodic measure $\tau$ of $f$ it follows that 
\[
\tau\big(B_{M}(f^{-1}(x_0),\frac{1}{i})\big) \geq \frac{1}{L(\frac{1}{i})}>0.
\]
Set $\delta(i) = \frac{1}{L(\frac{1}{i})}$. We define 
$$\beta_{-1}=1,  \,\, \beta_{i-1}:= \frac{l_{i_0 + i}}{i_0 +i} \cdot \delta ({i_0 + i}),\,\,\,\,\forall  i \geq 1.$$
By Lemma \ref{exifunofpr}, one can find a $C^{\infty}$ function $\omega: \xi(\tilde{V}) \rightarrow [0,1]$ satisfying the following properties:
\begin{enumerate}
	\item $\omega|_{\xi(\tilde{V}\setminus V)} \equiv 1$;
	\item $|\omega|_{B^{5}(\frac{1}{i_0 + i})}|_{C^{0}} \leq \beta_{i-1}$;
	\item $\omega(0) =0$ and $0 < \omega(a) \leq 1$ for $0 \neq a \in \xi(\tilde{V})$.
\end{enumerate}
 
We then define a function $\alpha \in C^{\infty}(\Omega,[0,1])$ as follows:
\begin{equation}\label{defalpha}
	\alpha(q)=\left\{
	\begin{aligned}
		\omega \circ \xi(q) & , & q\in \tilde{V}\\
		1 & , & q \notin \tilde{V}.
	\end{aligned}
	\right.
\end{equation}
Let $Y= \alpha X$ and $\phi _t$ be the flow induced by $Y$.

For any $x\in M,$ define $\gamma(x)$ by:
	\begin{equation}
			\left\{
			\begin{aligned}
				&\phi_{\gamma(x)}\big((x,0)\big) = \psi_{1}\big(x, 0\big) = \big(f(x),0\big)\quad &(x,0) \neq \psi_{1}^{-1}(p) \text{ or } (x,0) \neq p; \\
				&\gamma(x) = +\infty & (x,0) = \psi _{1}^{-1}(p) \text{ or } (x,0) = p.
				\end{aligned}
		\right.
		\end{equation}
For any ergodic measure $\mu$ of $f$, denote by $E_{\mu} = \int_{M}\gamma(y) d \mu(y)$.

The following lemma provides a sufficient condition for a system to  have only atomic  measures.

\begin{lem}[Corollary 2.12 of \cite{S-Y-Z}] \label{sufficientcondition}
	If $E_{\mu}(\gamma) = + \infty$ for all non-atomic ergodic measure $\mu$ of $f$, then $\phi_t$ has only atomic invariant probability measures. 
\end{lem} 

Observe that for $x\in B_{M}(f^{-1}(x_0), \frac{1}{i_2+i})$, we have
\[
l_{i_0 + i} = \int_{t(x)}^{\gamma(x)} \sqrt{<\alpha(\phi_s(x))X(\phi_s(x)), \alpha(\phi_s(x))X(\phi_s(x))>} ds,
\]
where $t(x)>0$ and $\phi_{t(x)} (x) = \psi_{1-l_{i_0 +i}}(x)$. Then 
\[
\gamma(x) \geq \gamma(x) - t(x) \geq \frac{l_{i_0 + i}}{|\alpha|_{B_{\Omega}(p, \frac{1}{i_2 + i})}|\,\norm{X}} \geq \frac{l_{i_0 +i}}{\beta_{i-1}\norm{X}} = \frac{i_0 +i}{\delta(i_0 +i)\norm{X}}
\]
for any $x\in B_{M}(f^{-1}(x_0), \frac{1}{i_0 + i}).$ Therefore
\[
\gamma|_{B_{M}(f^{-1}(x_0), \frac{1}{i_2+i})} \geq \frac{i_0 +i}{\delta(i_0 +i)\norm{X}}.
\]

Now for any ergodic measure $\tau$, we have
\[
E_{\tau}(\gamma) = \int_{M}\gamma(y) d \tau(y) \geq  \frac{i_0 +i}{\delta(i_0 +i)\norm{X}} \tau(B_{M}(f^{-1}(x_0),\frac{1}{i_0 +i}) \geq  \frac{i_0 +i}{\norm{X}} \rightarrow \infty,
\]
as $i \rightarrow \infty$. By Lemma \ref{sufficientcondition}, all ergodic measures of $\phi$ is atomic. Thus  $\phi$ has Dirac measure at $p$ as its unique ergodic measure. 
Since the original system $f: M \rightarrow M$ is minimal and thus $\psi$ is minimal,  the $\phi$ orbit of any point outside $Orb(p, \psi)$  will  return to  arbitrarily small   neighborhood of the point $p$ infinite times, which implies that $\phi$ is transitive. So $\phi$ is topologically chaotic but statistically trival.

{\bf Step 2.}\, Conected set of vector fields   which are  topologically chaotic but statistically trivial.

Denote ~$\Omega$~ in Step 1 by ~$M$~ and denote ~$Y$~in  Equation  (\ref {defalpha})    by ~$X.$~ 
Denote by $\phi$ the flow indeced by $X.$

Recall  ${\cal{T}}^r(M)$ denote the set of  $C^r$ vector fields which are topologically chaotic but statistically trivial.
Then  $X\in {\cal{T}}^r(M),\,\, r\geq 1$ by Step 1. 
Denote $m = \dim M=5$. Recall  by $B^m(k)$  we denote the open ball in $\mathbb{R}^m$ with radius $k$ centered at $0$. Take a local card $(\tilde{V},\xi)$ around $p=(x_0, 0)$ 
in condition ({\bfseries H})
satisfying the following properties:
\begin{itemize}
\item[(1)] There is an open set $V$ in $M$ such that
$$p\in V\subseteq \overline{V}\subseteq\tilde{V};$$
\item[(2)]$\xi(p)=0,\,\xi(V)=B^m(1),\,\xi(\tilde{V})=B^m(2);$
\item[(3)]$D\xi\circ X\circ D\xi^{-1}$ is a system of differential equations
$$
\left\{
\begin{aligned}
\frac{dx_1}{dt}&=&f_1(x),\\
&\cdots&  \\
\frac{dx_m}{dt}&=&f_m(x),
\end{aligned}
\right.
$$ where $x\in B^m(2)$.
\end{itemize}

Given an integer $\ell> 1$, take $\beta_{2\ell}\in C^\infty(\mathbb{R},\mathbb{R})$ such that

\begin{align*}
\left\{
\begin{aligned}
\beta_{2\ell}(\tau)&=\tau^{2\ell},& &\tau\in(-\frac{1}{2},\frac{1}{2})\\
\beta_{2\ell}(\tau)&=1,& &\tau\in\mathbb{R}\setminus(-1,1).
\end{aligned}
\right.
\end{align*}

Set $\alpha_{2\ell}: \mathbb{R}^m\to\mathbb{R}$, $\alpha_{2\ell}(x) =\beta_{2\ell}(\|x\|)$, and set $g_i^{2\ell}(x) = \alpha_{2\ell}(x)f_i(x)$ and denote the system of differential equations
$$
\left\{
\begin{aligned}
\frac{dx_1}{dt}&=&g_1^{2\ell}(x),\\
&\cdots&  \\
\frac{dx_m}{dt}&=&g_m^{2\ell}(x),
\end{aligned}
\right.
$$
where $x\in B^m(2)$ by $Y'_{2\ell}$, $\ell=1,2,3,\cdots$. From the construction ~$Y'_2,\,Y'_4, Y'_6, \cdots$~ are linearly independent. If we denote
$$
Y_{2\ell}(x):=\left\{
\begin{aligned}
&D\xi^{-1}(Y'_{2\ell}),& &x\in\tilde{V}\\
&X(x),& &x\in M\setminus\tilde{V},
\end{aligned}
\right.
$$
then $Y_2,\,Y_4, Y_6, \cdots$ are linearly independent. 
For $i = 1, \cdots, m$, take a $C^\infty$ function $\alpha'_i:\,\mathbb{R}^m\to \mathbb{R}$ such that
\begin{itemize}
\item[(1)] $0\leq\alpha'_i\leq\alpha_{2\ell}(x)$ for some $\ell$; and 
\item[(2)] $\alpha'_i(x)=0\iff x=0.$ 
\end{itemize}
Set
$$Y'=(\alpha'_1(x)f_1(x),\cdots,\alpha'_m(x)f_m(x))^{tr}.$$
Let ${\cal{A}}_1(X)$ denote the set of all vector fields defined as follows:
$$
Y(x):=\left\{
\begin{aligned}
&D\xi^{-1}(Y'),& &x\in\tilde{V}\\
&X(x),& &x\in M\setminus\tilde{V}.
\end{aligned}
\right.
$$
By taking arbitrarily ~$\alpha'=(\alpha_1',\cdots, \alpha'_m)$~the set ~${\cal{A}}_1(X)$~ is  infinite. It is clear that
$Y_{2\ell}\in{\cal{A}}_1(X)$ for $\ell = 1, 2, 3, \cdots$ and thus  $\dim{\cal{A}}_1(X) = \infty$.
Now we take a family of $C^\infty$ functions $\beta^\theta:\,\mathbb{R}\to\mathbb{R}$, $\theta\in(0, 1)$ such that
\begin{itemize}
\item[(1)] 
$$
\left\{
\begin{aligned}
\beta^\theta(\tau)&<\tau^{2},& &\tau\in(-\frac{\theta}{2},\frac{\theta}{2})\\
\beta^\theta(\tau)&=1,& &\tau\in\mathbb{R}\setminus(-\theta,\theta);
\end{aligned}
\right.
$$
\item[(2)] $\beta^\theta(\tau)=0\iff  \tau=0$;
\item[(3)]	$\beta^{\theta_1}(\tau)<\beta^{\theta_2}(\tau),\,\forall \ \theta_1<\theta_2,\,\tau\in(-\frac{\theta_1}{2},\frac{\theta_1}{2})$;
\item[(4)] $\lim\limits_{\tau\to0}\lim\limits_{\theta\to0}\beta^\theta(\tau)=1$.
\end{itemize}
The existence of such family is clear. It is also clear that
$$
\lim_{\theta\to0}\beta^\theta(\tau)=\left\{
\begin{aligned}
0, \quad &\tau=0\\
1, \quad&\tau\neq0.
\end{aligned}
\right.
$$

Set $\alpha^\theta(x)=\beta^\theta(\|x\|)$ and $Y'^\theta(x)=(\alpha^\theta(x)f_1(x),\cdots,\alpha^\theta(x)f_m(x))^{tr},$ $x\in\mathbb{R}^m$ and define 
$$
Y^\theta(x):=\left\{
\begin{aligned}
&D\xi^{-1}(Y'^\theta),& &x\in\tilde{V}\\
&X(x),& &x\in M\setminus\tilde{V}.
\end{aligned}
\right.
$$

From the construction, it is clear that $Y^\theta\in {\cal{A}}_1(X)$. 
Denote by ~${\cal{A}}_2(X)   $~the set consisting of $Y^\theta(x), $ $\theta\in (0,1).$ 
From $Y^\theta$ we get the vector field $X$ in Step 1  by letting $\theta\to 0$, that is, $X = Y^0$. Denote ${\cal{A}}(X) = {\cal{A}}_2(X)\cup\{X\}$. Then ${\cal{A}}(X)$ is
a connected set with  $\infty$ dimension.

Next we show that ${\cal{A}}(X) \subset {\cal{T}}^r(M)$. For given $Y\in  {\cal{A}}(X),$
let $\psi$ be the flow induced by $Y$. Since $X$ is topologically chaotic and statistically trivial,  by Lemma \ref{lem.zhou}, its minimal attracting center is exact $\{p\}$. So
\begin{align}
\lim_{\delta\to0}\lim_{t\to+\infty}\frac{1}{t}\int_{0}^t\chi_{B(p,\delta)}(\phi_sx)ds=1,\,\forall x\in M. \label{est1}
\end{align}

We need two claims.

\begin{claim}\label{claim. hit time }
	Let  $\psi$ be the  flow induced by $Y\in {\cal{A}}(X) .$ For any given $\delta_2>0$ small, we can pick  $0< \delta_1 < \delta_{2}$  	
	such that for any point in $\partial B(p,\delta_1)$, its first hitting time to $\partial B(p,\delta_{2})$ is always greater than $2$.  
\end{claim}

{\bf Proof of claim  \ref {claim. hit time }   }

Denote ~$\| Y\|=\sup_{x\in M} \| Y(x)\|.  $~ Then  ~$\| Y\|>0.$~ For given $t_1 ,\,\, t_2$ small with $0<t_2-t_1<\frac 12,$ 
take ~$x\in \partial B(p,\delta_1) $~ and denote ~$s$~ the first  hiting time on  $\partial B(p,\delta_{2}),$ that is, ~$   \psi_s(x) \in \partial B(p,\delta_{2})    $~ and ~$ \psi_t(x) \not\in \partial B(p,\delta_{2}),\,\, 0<t<s.   $~ 
Then
$$ \delta_2-\delta_1\leq dist (x, \psi_s(x)     ) = \int_0^s \| Y(\psi_t(x))   \|    \, dt =    \| Y(\psi_r(  x))   \|  s, \,\,\,\, \text{ where } \,\, 0<r<s,  $$
$$ s\geq \frac { \delta_2-\delta_1  } { \| Y(\psi_r(  x))   \|    }.     $$
By construction  $ \| Y (\psi_r(  x))    \|$ goes to zero    exponentially  as $\tau$   going to zero, 
so $s$ is bigger than 2 for small $t_1 ,\,\, t_2$.\\

\begin{claim}\label{claim.attracting center}
Let $\psi$ be  the flow induced by $Y   \in  {\cal{A}}(X) $. We have 
\begin{align}
\lim_{\delta\to0}\liminf_{t\to+\infty}\frac{1}{t}\int_{0}^t\chi_{B(p,\delta)}(\psi_sx)ds=1,\,\forall x\in M. \label{est2}
\end{align}
\end{claim}

{\bf Proof  of claim \ref{claim.attracting center}}


If $x=p$, the conclusion naturally holds. So we will prove the claim in the case of $x\in M\setminus \{p\}$.

Pick  $0< \delta_1 < \delta_{2}$ such that for any point in $\partial B(p,\delta_1)$, its first hit time to $\partial B(p,\delta_{2})$ is always greater than $2$. If there is $s_0$ such that
$$\psi_{s_0+t}(x)\in B(p,\delta_{2}), \forall t>0,$$
then (\ref{est2}) holds  by taking ~$\delta_2$~small in  claim \ref {claim. hit time }. So we will focus on the case that the orbit $\{\psi_t(x)\}$ enters $B(p,\delta_{1})$     and leaves  $B(p,\delta_{2})$ for infinitely many times. 

By Lemma \ref{lem:1} there exists a reparameterized time $\theta(x,t)$ transfing $\phi$ to $\psi$ which is continuous in $ M\setminus \{p\}$. According to the definition, see  Equation $(\ref{defalpha})$,  $Y(\psi)$ preserves the orbit of $X(\phi)$ and keeps its orientation. So the transfer map $\pi$ is the identity.  Consequently, we have 
$$\phi_{t}(x) = \psi_{\theta(x,t)}(x),$$ 
that is, $\psi$ is a time-changed flow from $\phi$ by reparameterization $\theta(x,t).$
Due to Equation $(\ref{defalpha})$, it holds that $\|Y(x)\|\leq\|X(x)\|$ for each $Y\in  {\cal{A}}(X)$. Combining with Proposition \ref{pro:3}, we have
\begin{align}
\liminf\limits_{t\rightarrow \infty} \frac{\theta(x,t)}{t} \geq 1. \label{est4}
\end{align}
 By the continuity of $\theta$, there exists $C>0$ such that $\theta(x,t) \leq C$ for  $0\leq t \leq 2$ and $x\in M\setminus B(p,\delta_1)$. For any real number $\tau \geq 2$, write $\tau = r+l$, where $r\in \mathbb{N}$ and $1< l < 2.$ If $\phi_{[0,\tau)}(x)\subset M\setminus B(p,\delta_1)$, then
\begin{align}
\theta(x,\tau) = \theta(x,1) +\theta(\phi_{1}(x),1) +\cdots+\theta(\phi_{r-1}(x),1)+
 \theta(\phi_{r}(x), l) \leq C(r+1) \leq C\tau. \label{est5}
\end{align}

For $x\in M\setminus \overline{B}(p,\delta_{2})$, take $a_1$ as its first hit time to $ \partial B(p,\delta_2)$, that is $$\phi_{[0,a_1)}(x) \cap B(p,\delta_2) = \emptyset \quad \text{ and } \phi_{a_1}(x) \in \partial B(p,\delta_2).$$ From (\ref{est1}), we know the orbit starts at $\phi_{a_1}(x)$ will stay in $B(p,\delta_1)$ for ``most of the time". Thus there is $b_1 > a_1$ such that $$\phi_{[a_1,b_1)} (x)\cap B(p,\delta_1) = \emptyset\quad \text{ and }\quad \phi_{b_1}(x) \in \partial B(p,\delta_1).$$  By transitivity, there exists $b_2 > b_1$ such that $$\phi_{[b_1,b_2]}(x) \subset \overline{B}(p,\delta_1)\quad \text{ and }\quad \phi_{b_2 + r}(x) \cap B(p,\delta_1) = \emptyset\quad \text{ for }r>0 \text{ small}.$$  We also can
take $a_2>b_2$ such that $$\phi_{[b_2,a_2]} \subset \overline{B}(p,\delta_{2})\text{ and }\phi_{(a_2,a_2+r']} \cap B(p,\delta_{2}) = \emptyset\text{ for }r'>0 \text{ small}.$$ Take $a_3>a_2$ such that $$\phi_{[a_2,a_3)} \cap B(p,\delta_2) = \emptyset\quad \text{ and }\quad \phi_{a_3}(x) \in \partial B(p,\delta_2).$$ Take $b_3>a_3$ such that $$\phi_{[a_3,b_3)} \cap B(p,\delta_1) = \emptyset\quad \text{ and }\quad \phi_{b_3}(x) \in \partial B(p,\delta_1).$$ Take $b_4 >a_3$ such that  $$\phi_{[b_3,b_4]}(x) \subset \overline{B}(p,\delta_1)\quad \text{ and }\quad \phi_{b_4 + r''}(x) \cap B(p,\delta_1) = \emptyset \text{ for } r''>0\text{ small}.$$  Take $a_4 > b_4$ such that
$$\phi_{[b_4,a_4]} \subset \overline{B}(p,\delta_{2})\quad \text{ and }\quad \phi_{(a_4,a_4+r''']} \cap B(p,\delta_{2}) = \emptyset \text{ for } r'''>0\text{ small}.$$ The above argument demonstrates that once the existence of $a_1, b_1, b_2, a_2$ is established, we can then choose appropriate $a_3, b_3, b_4, a_4$. By repeating this procedure, we can have the following sequence
\[
a_1 < b_1 <b_2<a_2 < a_3 < b_3<b_4<a_4 < a_5<b_5<b_6<a_6<\cdots
\] 
such that $\phi_{[a_{2k-1}, a_{2k}]}(x) \subset \xoverline{B}(p,\delta_{2})$ and 
$\phi_{(a_{2k}, a_{2k+1})}(x) \cap  \xoverline{B}(p,\delta_{2}) = \emptyset$ and $\phi_{[b_{2k-1}, b_{2k}]}(x) \subset \xoverline{B}(p,\delta_{1})$
and $\phi_{(b_{2k}, b_{2k+1})}(x) \cap \xoverline{B}(p,\delta_1) = \emptyset$, $k\geq 1$.

\vspace*{2pt}

\begin{center}\
	\tikzset{every picture/.style={line width=0.95pt}} 
\begin{tikzpicture}[x=0.55pt,y=0.55pt,yscale=-1,xscale=1]
	
	\draw   (225,137) .. controls (225,83.43) and (268.43,40) .. (322,40) .. controls (375.57,40) and (419,83.43) .. (419,137) .. controls (419,190.57) and (375.57,234) .. (322,234) .. controls (268.43,234) and (225,190.57) .. (225,137) -- cycle ;
	\draw   (275.48,137) .. controls (275.48,110.76) and (296.76,89.48) .. (323,89.48) .. controls (349.24,89.48) and (370.52,110.76) .. (370.52,137) .. controls (370.52,163.24) and (349.24,184.52) .. (323,184.52) .. controls (296.76,184.52) and (275.48,163.24) .. (275.48,137) -- cycle ;
	\draw    (325.33,161.15) .. controls (330.67,168.48) and (331.67,197.82) .. (347.67,206.97) .. controls (363.67,216.12) and (451.89,247.23) .. (470.17,213.47) .. controls (488.44,179.71) and (363.67,122.97) .. (345.67,132.97) .. controls (327.67,142.97) and (312.67,113.97) .. (335.17,102.47) .. controls (357.67,90.97) and (373.33,68.94) .. (388.17,50.47) .. controls (392.05,45.63) and (395.48,40.62) .. (399.67,35.99) .. controls (411.45,22.97) and (429.24,13.06) .. (480.17,18.97) ;
	
	\draw (252.56,130.6) node [anchor=north west][inner sep=0.75pt]  [rotate=-0.48] [align=left] {$\displaystyle \delta _{1}$};
	\draw (207.56,86.93) node [anchor=north west][inner sep=0.75pt]  [rotate=-0.48] [align=left] {$\displaystyle \delta _{2}$};
	\draw (397.06,48.93) node [anchor=north west][inner sep=0.75pt]  [font=\scriptsize,color={rgb, 255:red, 208; green, 2; blue, 27 }  ,opacity=1 ,rotate=-0.48] [align=left] {$\displaystyle \phi _{a_{1}}( x)$};
	\draw (312.89,68.93) node [anchor=north west][inner sep=0.75pt]  [font=\scriptsize,color={rgb, 255:red, 74; green, 144; blue, 226 }  ,opacity=1 ,rotate=-0.48] [align=left] {$\displaystyle \phi _{b_{1}}( x)$};
	\draw (371.89,118.43) node [anchor=north west][inner sep=0.75pt]  [font=\scriptsize,color={rgb, 255:red, 74; green, 144; blue, 226 }  ,opacity=1 ,rotate=-0.48] [align=left] {$\displaystyle \phi _{b_{2}}( x)$};
	\draw (440.06,146.93) node [anchor=north west][inner sep=0.75pt]  [font=\scriptsize,color={rgb, 255:red, 208; green, 2; blue, 27 }  ,opacity=1 ,rotate=-0.48] [align=left] {$\displaystyle \phi _{a_{2}}( x)$};
	\draw (371.06,233.43) node [anchor=north west][inner sep=0.75pt]  [font=\scriptsize,color={rgb, 255:red, 208; green, 2; blue, 27 }  ,opacity=1 ,rotate=-0.48] [align=left] {$\displaystyle \phi _{a_{3}}( x)$};
	\draw (346.39,181.93) node [anchor=north west][inner sep=0.75pt]  [font=\scriptsize,color={rgb, 255:red, 74; green, 144; blue, 226 }  ,opacity=1 ,rotate=-0.48] [align=left] {$\displaystyle \phi _{b_{3}}( x)$};
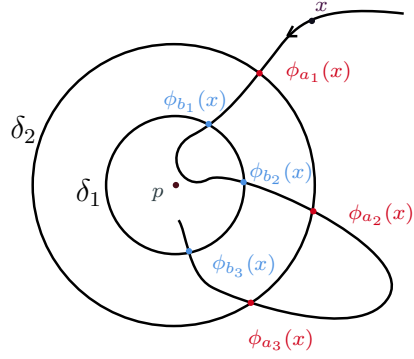
	\draw (228.5,257) node [anchor=north west][inner sep=0.75pt]   [align=left] {{\bf Fig 1.}  Choice of $\displaystyle a_{1} ,\ b_{1} ,\ b_{2} ,a_{2} \ \cdots $};
	
	\fill [color={rgb, 255:red, 208; green, 2; blue, 27 }] (381,60) circle (1.2pt);
	\fill [color={rgb, 255:red, 74; green, 144; blue, 226}] (346,95) circle (1.2pt);
		\fill [color={rgb, 255:red, 74; green, 144; blue, 226}] (370,135) circle (1.2pt);
			\fill [color={rgb, 255:red, 208; green, 2; blue, 27 }] (418,155) circle (1.2pt);
			\fill [color={rgb, 255:red, 208; green, 2; blue, 27 }] (375,218) circle (1.2pt);
\fill [color={rgb, 255:red, 74; green, 144; blue, 226}] (332.5,182.5) circle (1.2pt);

	\draw (417.06,8.53) node [anchor=north west][inner sep=0.75pt]  [font=\scriptsize,color={rgb, 225:red, 58; green, 12; blue, 57 }  ,opacity=1 ,rotate=-0.48] [align=left] {$\displaystyle x$};
		\fill [color={rgb, 255:red, 14; green, 14; blue, 26}] (417,24) circle (1.2pt);
	\draw (305,137) node [anchor=north west][inner sep=0.75pt]  [font=\scriptsize,color={rgb, 225:red, 38; green, 52; blue, 57 }  ,opacity=1 ,rotate=-0.48] [align=left] {$\displaystyle p$};
	\fill [color={rgb, 255:red, 74; green, 14; blue, 26}] (323.5,137) circle (1.2pt);

\draw[-] (400.1,34.5)--(404, 27.5);
\draw[-] (400.1,34.5)--(407, 34.5);
\end{tikzpicture}
\end{center}

Notice that 
\[
\phi_{a_k}(x) = \psi_{\theta(x,a_k)}(x), \ \phi_{b_k}(x) = \psi_{\theta(x,b_k)}(x), \ \forall k\geq 1.
\]
By (\ref{est5}), we have
\begin{align}
\theta(x, b_{2k+1}) - \theta(x,b_{2k})  = \theta(\phi_{b_{2k}}(x), b_{2k+1} -b_{2k}) \leq C (b_{2k+1} -b_{2k}), \forall k\geq 1. \label{est6}
\end{align}
For any $x\neq  p$, we  pick  small $\delta_1$ and $\delta_{2}$. By the definition of the sequence 
$$\{a_1,b_1,b_2,a_2,a_3, b_3, b_4, a_5,\cdots\},$$  we have
\begin{align}
&\,\,\,\,\,\,  \liminf_{t\to +\infty}\frac {1}{t} \int_0^t \chi _{B(p,\delta_1)}(\psi_sx) ds\\ 
&= \liminf_{n\to +\infty}\frac {1}{\theta(x,b_{2n+1})} \int_0^{\theta(x,b_{2n+1})} \chi _{B(p,\delta_1)}(\psi_sx) ds\notag \\
&=1- \limsup\limits_{n\rightarrow \infty}\frac{\theta(x,b_3) - \theta(x,b_2) + \cdots + \theta(x,b_{2n+1})-\theta(x,b_{2n}) }{\theta(x,b_{2n+1})}\notag\\
\end{align}
Hence, by (\ref{est6}), it holds that

\begin{align}
	&\,\,\,\, | \liminf_{t\to +\infty}\frac {1}{t} \int_0^t \chi _{B(p,\delta_1)}(\psi_sx) ds -1|\\
&\leq\limsup\limits_{n\rightarrow \infty}\frac{C(b_3-b_2 + \cdots + b_{2n+1} - b_{2n})}{\theta(x,b_{2n+1})}\notag\\
&= \limsup\limits_{n\rightarrow \infty} \frac{C(b_3-b_2 + \cdots + b_{2n+1} - b_{2n})}{b_{2n+1}} \cdot\frac{b_{2n+1}}{\theta(x,b_{2n+1})}\notag
\end{align}
Since $\phi$ and $\psi$  share the same orbit, it holds that 
$\phi_{(b_{2k}, b_{2k+1})}(x) \cap \xoverline{B}(p,\delta_1) = \emptyset,\,\,\,\,k\geq 1$. 
By using (\ref{est1})  we have 
$$\limsup\limits_{n\rightarrow \infty} \frac{b_3-b_2 + \cdots + b_{2n+1} - b_{2n}}{b_{2n+1}}$$
tends  to zero, provided $\delta_1\to 0.$ 
Thus combining with  (\ref{est4}), we have
\begin{align*}
\lim_{\delta\to0}\liminf_{t\to+\infty}\frac{1}{t}\int_{0}^t\chi_{B(p,\delta)}(\psi_sx)ds=1,\,\forall x\in M. \label{est2}
\end{align*}
Then we have completed the proof of Claim \ref{claim.attracting center}.

Claim \ref{claim.attracting center} states that the singleton set $\{p\}$ is a minimal  attracting center for $\psi$. By Lemma \ref{lem.zhou}, $Y$ is statistically trivial. Observe
that the orbits of $Y$ and $X$ coincide  and $X$ is transitive, so $Y$ is topologically transitive  as well as $X$. Hence we have proved that every $Y\in {\cal{A}}(X)$ is  topologically chaotic and statistically trivial.

\hfill $\blacksquare$

\begin{rk}

A similar argument works for the case when 
$\mathcal {M}_{\phi, erg}= $ \{{\itshape finitely many  fixed points}\}.
We propose a  problem:  
whether the result is true or not in the situation that $\mathcal{M}_{\phi, erg}$ = \{{\itshape  infinitely  many  fixed points}\}?

\end{rk}




\section*{Acknowledgment}
Liang is supported by NSFC (\#12271538, \#12071328, \#11871487) and Program for Innovation Research in Central University in Finance and Economics. Ren is supported by NSFC (\#11801261, \# 12071285) and Ren is supported by  Key Laboratory of Nonlinear Analysis and its Applications (Chongqing University), Ministry of Education and also Ren would like to thank the Key Laboratory of Mathematics and Applied Mathematics, Peking University for the support while he was visiting School of Mathematical Sciences, Peking University;  Sun is supported by FAPSP and NSFC (\# 12090010, \# 12093157012, \# 11831001). Vargas is supported by FAPSP.

\bibliographystyle{amsplain}

{}

{}

{}

{}

\end{document}